\documentclass[reqno]{amsart}

\usepackage{amssymb, amsmath, amsthm}
\usepackage{color}
 \usepackage[bookmarks=false]{hyperref}

\newcommand*{\fplus}{\genfrac{}{}{0pt}{}{}{+}}
\newcommand*{\fdots}{\genfrac{}{}{0pt}{}{}{\cdots}}
\newcommand*{\fminus}{\genfrac{}{}{0pt}{}{}{-}}

\newcommand{\qrfac}[2]{{\left({#1}; q\right)_{#2}}} 
\newcommand{\pqrfac}[3]{{\left({#1};#3\right)_{#2}}}

\renewcommand{\Im}{\operatorname{Im}}

\newmuskip\pFqskip
\mathchardef\pFcomma=\mathcode`, 
\newcommand*\rPhis[6]{%
  \begingroup
  \begingroup\lccode`~=`,
    \lowercase{\endgroup\def~}{\pFcomma\mkern\pFqskip}%
  \mathcode`,=\string"8000
  {}_{#1}\phi_{#2}\Biggl[\genfrac..{0pt}{}{#3}{#4};#6, #5\Biggr]%
  \endgroup
}

\newtheorem{Theorem}{Theorem}[section]

\newtheorem{Corollary}[Theorem]{Corollary} 

\newtheorem*{rems}{Remarks} 
\newenvironment{Remarks}{\begin{rems}\normalfont}{\end{rems}}

\newtheorem*{rem}{Remark} 
\newenvironment{Remark}{\begin{rem}\normalfont}{\end{rem}}

\numberwithin{equation}{section}

\dedicatory{Dedicated to Bruce Berndt on the occasion
of his 80th birthday}

\begin{document} 
\title[Ramanujan's Entry II.16.12]{On Entry~II.16.12:\\ A continued fraction of Ramanujan}

\author[G.~Bhatnagar]{Gaurav Bhatnagar
}
\address{Fakult\"at f\"ur Mathematik,  Universit\"at Wien \\
Oskar-Morgenstern-Platz 1, 1090 Wien, Austria.}

\email{bhatnagarg@gmail.com}

\author[M.~E.~H.~Ismail]{Mourad~E.~H.~Ismail}
\address{Department of Mathematics, University of Central Florida, Orlando, FL 32816, USA  }
\email{ismail@math.ucf.edu}

\date{\today}

\begin{abstract}
We study a continued fraction due to Ramanujan, that he recorded as Entry 12 in Chapter 16 of his second notebook. It is presented in Part III of Berndt's volumes on Ramanujan's notebooks. 
 We give two  alternate approaches to proving Ramanujan's Entry 12, one using a method of Euler, and another using the theory of orthogonal polynomials. We consider a natural generalization of Entry 12 suggested by the theory of orthogonal polynomials. 
\end{abstract}

%


\maketitle

%
%

\section{Introduction} 
Ramanujan's Entry 12 in Chapter 16, Notebook II is stated as follows. 
Let $|q|<1,$ and $|ab|<1$. Then we have
(see Berndt \cite[p.\ 24]{Berndt1991-RN3}):
\begin{multline}
\frac{\pqrfac{a^2q^3, b^2q^3}{\infty}{q^4}}{\pqrfac{a^2q, b^2q}{\infty}{q^4}}
=\cr
\frac{1}{1-ab}\fplus\frac{(a-bq)(b-aq)}{(1-ab)(1+q^2)}\fplus 
\frac{(a-bq^3)(b-aq^3)}{(1-ab)(1+q^4)}\fplus
\frac{(a-bq^5)(b-aq^5)}{(1-ab)(1+q^6)}\fplus
\fdots. 
\label{entry12}
\end{multline}
G.~H.~Hardy \cite[p.~XXV]{Ramanujan-CW} remarked: 
\begin{quote}
There is always more in one of Ramanujan's formulae than meets the eye, as anyone who gets to work to verify those which look the easiest will soon discover.
\end{quote}
Indeed, the verification of this beautiful continued fraction evaluation has proved to be most troublesome, and is due to the efforts of many mathematicians.  The first proof was given by Adiga, Berndt, Bhargava and Watson \cite{ABBW1985}, who further acknowledged the help they received from R.~A.~Askey and D.~M.~Bressoud.
The proof is rather complicated. 
A few years later, Jacobsen \cite{LJ1989} clarified and simplified this proof. Ramanathan \cite{KGR1987a} gave another formal proof, via contiguous relations.  Our own proofs, and in this paper we give two, are also not without complication. 

The difficulties arise because the first term of Ramanujan's continued fraction is a bit off; it doesn't fit the pattern of the rest of the terms. Perhaps the continued fraction should have been 
$$
\frac{1}{2(1-ab)}\fplus\frac{(a-bq)(b-aq)}{(1-ab)(1+q^2)}\fplus 
\frac{(a-bq^3)(b-aq^3)}{(1-ab)(1+q^4)}\fplus\dots.
$$
Indeed, from the point of view of orthogonal polynomials, 
 the continued fraction to consider is:
\begin{multline}
\label{jfrac-e12}
H(x):= 
\frac{1-ab}{x(1-ab)+(1-ab)}\fplus\frac{(a-bq)(b-aq)}{x(1-ab)+(1-ab)q^2}\fplus 
\cr
\frac{(a-bq^3)(b-aq^3)}{x(1-ab)+(1-ab)q^4}\fplus
\frac{(a-bq^5)(b-aq^5)}{x(1-ab)+(1-ab)q^6}\fplus
\fdots. 
\end{multline}

This is an example of   $J$-fraction, that is, a continued fraction of the form
\begin{equation*} 
\frac{A_0}{A_0x+B_0}\fminus\frac{C_1}{A_1x+B_1}\fminus\frac{C_2}{A_2x+B_2}\fminus\fdots.
\end{equation*} 
The $k$th convergent of this continued fraction is given by
$$\frac{N_k(x)}{D_k(x)} :=
\frac{A_0}{A_0x+B_0}\fminus\frac{C_1}{A_1x+B_1}\fminus\fdots \fminus \frac{C_{k-1}}{A_{k-1}x+B_{k-1}}.
$$
The numerators and denominators of the convergents are polynomials and satisfy a three-term recurrence relation. 

Using \cite[Th.~2.6.1, p.\ 35]{MI2009}, we see that 
 the recurrence relation corresponding to \eqref{jfrac-e12} is given by
\begin{multline}\label{entry12-3term}
y_{k+1}(x) = ((1-ab)x + (1-ab)q^{2k}) y_k(x) +
\cr ab(1-bq^{2k-1}/a)(1-aq^{2k-1}/b) y_{k-1}(x), \text{ for } k > 0.
\end{multline}
The numerator and denominator satisfy \eqref{entry12-3term}, with the initial values
$$N_0(x)=0, N_1(x)=1-ab;\   D_0(x)=1, D_1(x)=(1-ab)(x+1). $$ 
It is clear that $N_k(x)$ and $D_k(x)$ are polynomials, of degree $k-1$ and $k$, respectively. 
From the theory of orthogonal polynomials, it follows that the denominator polynomials are orthogonal with respect to a positive measure, provided we assume $a$ and $b$ are real numbers with different signs.

In a previous study, the first author \cite{GB2014} showed that an elementary method due to Euler 
can be used to derive all of Ramanujan's $q$-continued fractions, except for one notable exception. 
The exception was Entry 12. In Section~\ref{sec:euler}, we remedy this omission. The primary benefit 
of this approach is that it only relies on the $q$-binomial theorem. Further, the algebraic 
manipulations appear motivated. However, this approach only gives modified convergence (see Lorentzen and Waadeland~\cite{LW1992}) of the continued fraction. One needs some further restrictions on the parameters to obtain ordinary convergence. 

Next, in Section~\ref{sec:proof2}, we provide a proof of Entry 12 using what we may call the 
standard orthogonal polynomial approach. The proof requires asymptotic techniques and is obtained 
from the generating function of the numerator and denominator polynomials. Again, the proof is 
motivated, but rather surprisingly, one still needs the algebraic manipulations from Euler's method. This is because of the missing $2$ in the first denominator of \eqref{entry12}; this is what causes the complications. 
In both these proofs we can spot elements of the aforementioned proofs due to Berndt et.al., Ramanathan and Jacobsen. 

Finally, in Section~\ref{sec:stieltjes}, we compute the  value of the more general continued fraction $H(x)$,  and comment briefly on the orthogonality measure of the associated denominator polynomials. 

We have mentioned that Ramanujan's Entry 12 is slightly off. On examining our proof in 
Section~\ref{sec:proof2}, it appears that  Ramanujan would have been unable to write the continued 
fraction in terms of infinite products if 
the first denominator followed the pattern. Moreover, our proof of Section~\ref{sec:euler} begins 
with the products and shows how one can find the continued fraction expansion. If indeed this was similar to Ramanujan's own 
approach, it will explain why the first term of many of Ramanujan's continued fractions do not quite 
fit the pattern of the rest of the terms. 

%
%
%
%

We require the following definitions and notations from Lorentzen and Waadeland \cite{LW1992} and Gasper and Rahman~\cite{GR90}.
\begin{enumerate}
\item {\bf Continued fractions:} Let $\{ a_n\}_{n=1}^{\infty}$ and $\{ b_n\}_{n=0}^{\infty}$   be sequences of complex numbers, $a_n\neq 0$. Consider the sequence of mappings 
$$s_0(w)=b_0+w; s_n(w)=\frac{a_n}{b_n+w}, \text{ for } n=1, 2, \dots $$
These mappings are from ${\mathbb{C^*}}$ to ${\mathbb{C^*}}$, where ${\mathbb{C^*}}$ denotes the extended complex numbers ${\mathbb{C}}\cup \left\{\infty\right\}$.
Let $S_n(w)$ be defined as follows:
$$S_0(w)=s_0(w)=b_0+w; S_n(w)=S_{n-1}(s_n(w)).$$
Then $S_n(w)$ can be written as
$$S_n(w)=b_0+\frac{a_1}{b_1}\fplus\frac{a_2}{b_2}\fplus\fdots\fplus\frac{a_n}{b_n+w}.$$
Note that since $a_n\neq 0$, the $s_k(w)$ are non-singular linear fractional transformations.  The $S_n$ are compositions of these, and are thus also non-singular linear fractional transformations. 
 A continued fraction is an ordered pair
$\left( (\{ a_n\}, \{ b_n\}), S_n(0)\right)$, where $a_n$, $b_n$ and $S_n$ are as above. 
\item {\bf Convergence of continued fractions:} When $S_n(0)$ converges to an extended complex number $S$, we say that the continued fraction converges, and we write
$$S=b_0+\frac{a_1}{b_1}\fplus\frac{a_2}{b_2}\fplus\fdots.$$
The convergence of 
$S_n(w)$ (or even $S_n(w_n)$, for suitably chosen $w_n$) is called {\em modified convergence}. 
\item {\bf Factorial notation:} The {\em $q$-rising factorial} $\qrfac{a}{n}$ is defined as
$$\qrfac{a}{n} := 
\begin{cases}
1 & \text{ for } n=0\cr
(1-a)(1-aq)\cdots (1-aq^{n-1}) & \text{ for } n=1, 2, \dots.
\end{cases}
$$ In addition
$$\qrfac{a}{\infty} := \prod_{k=0}^\infty (1-aq^{k}) \text{ for } |q|<1.$$
We use the short-hand notation
\begin{align*}
\qrfac{a_1, a_2,\dots, a_r}{k} &:= \qrfac{a_1}{k} \qrfac{a_2}{k}\cdots
\qrfac{a_r}{k}.
\end{align*}
\item {\bf Basic hypergeometric series  (or $_r\phi_s$ series):}
 This series is of the form
\begin{equation*}
_{r}\phi_s \left[\begin{matrix} 
a_1,a_2,\dots,a_r \\
b_1,b_2,\dots,b_s\end{matrix} ; q, z
\right] :=
\sum_{k=0}^{\infty} \frac{\qrfac{a_1,a_2,\dots, a_r}{k}}{\qrfac{q, b_1,b_2,\dots, b_s}{k}}
\left( (-1)^kq^{\binom k2}\right)^{1+s-r} z^k.
\end{equation*}
When $r=s+1$, the series converges for $|z|<1$.
\end{enumerate}

\section{Entry 12 by Euler's method}\label{sec:euler}
In this section, we give a formal proof of  \eqref{entry12} using the approach  
Euler~\cite{LE1788-616} used to find the continued fraction of a divergent series. 
We use the elementary identity:
\begin{equation}\label{div-1step}
\frac{U}{V}=1+\frac{U-V}{V}
\end{equation}
to \lq divide' two series of the form $1+a_1z+a_2z^2+\cdots$. A repeated application leads to 
a continued fraction, much like Euclid's algorithm is used to find a continued fraction of a ratio 
of two numbers. 
We call this Euler's method. In \cite{GB2014}, all but one of Ramanujan's 
$q$-continued fractions from Chapter 16 of Ramanujan's second notebook \cite{Berndt1991-RN3} 
and from the Lost Notebook
\cite[Ch.\ 6]{AB2009} were proved using this method, or obtained as special cases. 
The only exception was Entry 12 of 
\cite[Ch.\ 16]{Berndt1991-RN3}. The purpose of this section is to rectify this omission.

Define, for $s=0, 1, 2, 3, \dots$
\begin{equation}\label{def-ds}
D(s):= \sum_{k=0}^{\infty}\frac{\pqrfac{bq^{2s-1}/a, -bq/a}{k}{q^2}}
{\pqrfac{q^2, -q^{2s}}{k}{q^2}}(a^2q)^{k}
.
\end{equation}


\begin{Theorem}\label{th:entry12-euler} Let $D(s)$ be as above. Then, for 
$|q|<1$ and $|a|<1$, and $s=0, 1, 2, 3, \dots$, we have
\begin{align}
&\frac{\pqrfac{a^2q^3, b^2q^3}{\infty}{q^4}}{\pqrfac{a^2q, b^2q}{\infty}{q^4}}
=\cr
\frac{1}{1-ab}\fplus &
\frac{(a-bq)(b-aq)}{(1-ab)(1+q^2)}\fplus 
\frac{(a-bq^3)(b-aq^3)}{(1-ab)(1+q^4)}\fplus \fdots \cr
&\fplus \frac{(a-bq^{2s-1})(b-aq^{2s-1})}{(1-ab)(1+q^{2s})}\fplus
\frac{(a-bq^{2s+1})(b-aq^{2s+1})}{(1+q^{2s+2})\cfrac{D(s+1)}{D(s+2)}}.
 \label{entry12-prop}
\end{align}
\end{Theorem}


\begin{Remark}
From Theorem~\ref{th:entry12-euler}, we immediately get modified convergence of the infinite continued fraction in Entry 12. For ordinary convergence, we need a further condition and a further argument. See \cite{GB2014} for an example in the context of Ramanujan's continued fractions. 
 \end{Remark}
\begin{proof}
We prove \eqref{entry12-prop} in three steps. First observe that
\begin{equation}\label{entry12-pf-step1}
\frac{\pqrfac{a^2q^3, b^2q^3}{\infty}{q^4}}{\pqrfac{a^2q, b^2q}{\infty}{q^4}}
=
 \frac{\displaystyle
\sum_{k=0}^\infty \frac{\pqrfac{bq/a, -bq/a}{k}{q^2} }{\pqrfac{q^2, -q^2}{k}{q^2}} (a^2q)^k
}
{\displaystyle
\sum_{k=0}^\infty \frac{\pqrfac{b/aq, -b/aq}{k}{q^2} }{\pqrfac{q^2, -q^2}{k}{q^2}} (a^2q^3)^k
} 
\end{equation}
To show this we use the $q$-binomial theorem \cite[Equation (I.3)]{GR90}. We can see that
\begin{align*}
\text{LHS of \eqref{entry12-pf-step1}} & =
\frac{\pqrfac{b^2q^3}{\infty}{q^4}/\pqrfac{a^2q}{\infty}{q^4}}
{\pqrfac{b^2q}{\infty}{q^4}/\pqrfac{a^2q^3}{\infty}{q^4}}
\cr
&=
\frac{\displaystyle
\sum_{k=0}^\infty \frac{\pqrfac{(bq/a)^2}{k}{q^4} }{\pqrfac{q^4}{k}{q^4}} (a^2q)^k
}
{\displaystyle
\sum_{k=0}^\infty \frac{\pqrfac{(b/aq)^2}{k}{q^4} }{\pqrfac{q^4}{k}{q^4}} (a^2q^3)^k
}
\end{align*}
which can be written as the right hand side of \eqref{entry12-pf-step1} using some elementary identities involving $q$-rising factorials. 

Secondly,  we show
 \begin{align}\label{entry12-pf-step2}
 \frac{\displaystyle\rPhis{2}{1}{b/aq, -{b}/{aq} }{-q^2}{a^2q^3}{q^2}}
 { \displaystyle \rPhis{2}{1}{ bq/a, -{bq}/{a}}{-q^2}{a^2q}{q^2}}
 = 
{1-ab} + \frac{(a-bq)(b-aq)}{(1+q^2) \cfrac{D(1)}{D(2)}}. 
\end{align}
To prove this we will use \eqref{div-1step}. Let $U$ (respectively, $V$) denote the $_2\phi_1$ sum
in the numerator (respectively, denominator) above.
Observe that
\begin{align*}
\frac{U}{V} & =\frac{
\displaystyle
\sum_{k=0}^\infty \frac{\pqrfac{b/aq, -b/aq}{k}{q^2} }{\pqrfac{q^2, -q^2}{k}{q^2}} (a^2q^3)^k
}
{\displaystyle
\sum_{k=0}^\infty \frac{\pqrfac{bq/a, -bq/a}{k}{q^2} }{\pqrfac{q^2, -q^2}{k}{q^2}} (a^2q)^k
}
\cr
&= 1 + \frac{1}{V}
\Bigg(
\displaystyle
\sum_{k=0}^\infty \frac{\pqrfac{b/aq, -b/aq}{k}{q^2} }{\pqrfac{q^2, -q^2}{k}{q^2}} (a^2q^3)^k
-
\displaystyle
\sum_{k=0}^\infty \frac{\pqrfac{bq/a, -bq/a}{k}{q^2} }{\pqrfac{q^2, -q^2}{k}{q^2}} (a^2q)^k
\Bigg) \cr
&=1 + \frac{1}{V}
\Bigg(
\displaystyle
\sum_{k=0}^\infty \frac{\pqrfac{b/aq, -b/aq}{k}{q^2} }{\pqrfac{q^2, -q^2}{k}{q^2}} (a^2q)^k
\left((-1)\frac{\left(1-q^{2k} \right)\left(1+b^2q^{2k}/a^2q^2 \right)}
{\left(1-b/aq \right)\left(1+b/aq \right)}
\right)
\Bigg).
\end{align*}
Note that the first term of the sum corresponding to $k=0$ is $0$. We cancel $\left(1-q^{2k} \right)$, and then shift the index to obtain 
\begin{align*}
\frac{U}{V}&=1 + \frac{(-1)}{V}
\Bigg(
\displaystyle
\sum_{k=1}^\infty \frac{\pqrfac{bq/a, -bq/a}{k-1}{q^2} }{\pqrfac{q^2}{k-1}{q^2}
\pqrfac{ -q^2}{k}{q^2}} (a^2q)^k
\left(1+b^2q^{2k}/a^2q^2 \right)
\Bigg)\cr
&=1 + \frac{(-1)}{V}
\Bigg(
\displaystyle
\sum_{k=0}^\infty \frac{\pqrfac{bq/a, -bq/a}{k}{q^2} }{\pqrfac{q^2, -q^2}{k}{q^2}} (a^2q)^k
\left[\frac {a^2q \left(1+b^2q^{2k}/a^2\right)}{1+q^{2k+2}} \right]
\Bigg).
\end{align*}
Next, we use
\begin{align}\label{star}\tag{$*$}
\frac {a^2q \left(1+b^2q^{2k}/a^2\right)}{1+q^{2k+2}} &= 
%
ab + \frac{ a\left(aq-b\right)\left(1-bq^{2k+1}/a\right)}{1+q^{2k+2}},
\end{align}
to obtain
\begin{align*}
\frac{U}{V}&=1+\frac{(-ab)}{V}\Bigg(
\displaystyle
\sum_{k=0}^\infty \frac{\pqrfac{bq/a, -bq/a}{k}{q^2} }{\pqrfac{q^2, -q^2}{k}{q^2}} (a^2q)^k
\Bigg) \cr
& - \frac{a(aq-b)(1-bq/a)}{\left(1+q^2\right) V}
\Bigg(
\displaystyle
\sum_{k=0}^\infty \frac{\pqrfac{bq^3/a, -bq/a}{k}{q^2} }{\pqrfac{q^2, -q^4}{k}{q^2}} (a^2q)^k
\Bigg)\cr
&= 1-ab + \frac{(a-bq)(b-aq)}{ 
\cfrac{\left(1+q^2\right) \Bigg(
\displaystyle
\sum_{k=0}^\infty \frac{\pqrfac{bq/a, -bq/a}{k}{q^2} }{\pqrfac{q^2, -q^2}{k}{q^2}} (a^2q)^k
\Bigg)}
{\Bigg(
\displaystyle
\sum_{k=0}^\infty \frac{\pqrfac{bq^3/a, -bq/a}{k}{q^2} }{\pqrfac{q^2, -q^4}{k}{q^2}} (a^2q)^k
\Bigg)}}
.
\end{align*}
In this manner, we obtain \eqref{entry12-pf-step2}.

%

Note that \eqref{entry12-pf-step1} and \eqref{entry12-pf-step2} give us
\eqref{entry12-prop} for $s=0$. The third fact we require is the recursion, for $s=0, 1, 2, \dots$,
\begin{align}\label{entry12-recursion}
(1+q^{2s})\frac{D(s)}{D(s+1)}=
(1-ab)(1+q^{2s})+\frac{(a-bq^{2s+1})(b-aq^{2s+1})}{(1+q^{2s+2}){\cfrac{D(s+1)}{D(s+2)}}}.
\end{align}
To prove \eqref{entry12-recursion}, we again apply \eqref{div-1step}. 
The steps are similar to those above. We have 
\begin{align*}
\frac{D(s)}{D(s+1)} & =\frac{
\displaystyle
\sum_{k=0}^\infty \frac{\pqrfac{bq^{2s-1}/a, -bq/a}{k}{q^2} }{\pqrfac{q^2, -q^{2s}}{k}{q^2}} (a^2q)^k
}
{\displaystyle
\sum_{k=0}^\infty \frac{\pqrfac{bq^{2s+1}/a, -bq/a}{k}{q^2} }{\pqrfac{q^2, -q^{2s+2}}{k}{q^2}} (a^2q)^k
}
\cr
&=1 + \frac{1}{D(s+1)}
\displaystyle
\sum_{k=0}^\infty \Bigg(
\frac{\pqrfac{bq^{2s-1}/a, -bq/a}{k}{q^2} }{\pqrfac{q^2, -q^{2s}}{k}{q^2}} (a^2q)^k 
\cr
&\hspace{1in} \times
\left((-1)\frac{\left(1-q^{2k} \right)\left(q^{2s}+bq^{2s-1}/a \right)}
{\left(1-bq^{2s-1}/a \right)\left(1+q^{2k+2s} \right)}
\right)\Bigg).
\end{align*}
Once again, we realize the first term is $0$, and after canceling $\left(1-q^{2k} \right)$, we shift the index so that it again begins from $0$. 
\begin{align*}
\frac{D(s)}{D(s+1)} 
&=
1 + \frac{(-1)}{D(s+1)}
\displaystyle
\sum_{k=0}^\infty \frac{\pqrfac{bq^{2s+1}/a, -bq/a}{k}{q^2} }{\pqrfac{q^2, -q^{2s+2}}{k}{q^2}
} (a^2q)^k \cr
&\hspace{1in}\times
\Bigg(
\left[\frac{\left( 1+bq^{2k+1}/a \right)\left(a^2q^{2s+1}+abq^{2s}\right)}
{\left(1+q^{2s}\right) \left(1+q^{2k+2s+2}\right)} \right]
\Bigg).
\end{align*}
%
We use
\begin{align}\label{2star}
\frac{\left( 1+bq^{2k+1}/a \right)\left(a^2q^{2s+1}+abq^{2s}\right)}
{\left(1+q^{2s}\right) \left(1+q^{2k+2s+2}\right)} 
&= ab +\frac{ a\left(aq^{2s+1}-b\right)\left(1-bq^{2k+2s+1}/a\right)}
{\left(1+q^{2s}\right) \left(1+q^{2k+2s+2}\right)},
\tag{$**$}
\end{align}
to obtain, after multiplying by $(1+q^2)$,
\begin{align*}
\left(1+q^2\right)\frac{D(s)}{D(s+1)}&=(1-ab)\left(1+q^{2s}\right)  
\cr
&\hspace{10pt}
- \frac{a(aq^{2s+1}-b)(1-bq^{2s+1}/a)}{\left(1+q^{2s+2}\right) D(s+1)} 
\cr
&\hspace{1in}\times
\Bigg(
\displaystyle
\sum_{k=0}^\infty \frac{\pqrfac{bq^{2s+3}/a, -bq/a}{k}{q^2} }{\pqrfac{q^{2}, -q^{2s+4}}{k}{q^2}} (a^2q)^k
\Bigg)
.
\end{align*}
This reduces to the right hand side of \eqref{entry12-recursion}.

Using \eqref{entry12-pf-step1} and \eqref{entry12-pf-step2}, we obtain \eqref{entry12-prop} for $s=0$. We iterate using \eqref{entry12-recursion} to complete the derivation.
\end{proof}
\begin{Remark} On the surface, it appears that the algebraic simplifications \eqref{star} and 
\eqref{2star} are unmotivated, and the difficult parts of the proof. However, the key difficulty in 
applying Euler's method is to write the continued fraction as a ratio of two series. There are 
multiple ways to do so, and not all are amenable to Euler's approach. 
Once this is done in 
\eqref{entry12-pf-step1}, the rest follows by applying Euler's method. 
Notice that if we ignore the quantity in square brackets in the sum in the expression before 
\eqref{2star}, the sum is $D(s+1)$. A similar remark applies to the sum above \eqref{star}. This 
observation motivates the algebraic manipulations required in the proof. Even so, the algebraic 
manipulations required for obtaining other continued fractions due to Ramanujan by this method 
(see \cite{GB2014}) are, comparatively, relatively straightforward.  
\end{Remark}

\section{A proof of Entry 12 using Darboux's Method}\label{sec:proof2}

The objective of this section to provide another proof of Ramanujan's Entry 12, by what may be called the standard orthogonal polynomial method.  Let $C$ denote Ramanujan's continued fraction.
$$C=\frac{1}{1-ab}\fplus\frac{(a-bq)(b-aq)}{(1-ab)(1+q^2)}\fplus 
\frac{(a-bq^3)(b-aq^3)}{(1-ab)(1+q^4)}\fplus
\fdots. 
$$
On the other hand, the continued fraction $K=H(1)/(1-ab)$ is
$$K=\frac{1}{2(1-ab)}\fplus\frac{(a-bq)(b-aq)}{(1-ab)(1+q^2)}\fplus 
\frac{(a-bq^3)(b-aq^3)}{(1-ab)(1+q^4)}\fplus
\fdots. 
$$
Evidently 
\begin{equation} \label{KC-relation}
\frac{1}{K}-(1-ab) = \frac{1}{C}.
\end{equation}
Our approach is to first find an expression for $H(1)$ and then use \eqref{KC-relation} to find the value of  $C$.

$$C=\frac{1}{1-ab}\fplus\frac{(a-bq)(b-aq)}{(1-ab)(1+q^2)}\fplus 
\frac{(a-bq^3)(b-aq^3)}{(1-ab)(1+q^4)}\fplus
\fdots
$$

$$K=\frac{H(1)}{1-ab}=\frac{1}{2(1-ab)}\fplus\frac{(a-bq)(b-aq)}{(1-ab)(1+q^2)}\fplus 
\frac{(a-bq^3)(b-aq^3)}{(1-ab)(1+q^4)}\fplus
\fdots
$$
\begin{equation*} \label{K-C-relation}
\frac{1}{K}-(1-ab) = \frac{1}{C}
\end{equation*}

\begin{Theorem} Let $|q|<1$, $|ab|<1$ and $|a^2q|<1$ and let $H(x)$ be the continued fraction \eqref{jfrac-e12}. Then
\begin{equation}\label{H1}
H(1) = \frac{(1-ab)}{2} \cdot
\frac{\displaystyle\rPhis{2}{1}{-{bq}/{a}, bq/a}{-q^2}{a^2q}{q^2}}
{ \displaystyle \rPhis{2}{1}{-{bq}/{a}, b/aq}{-1}{a^2q}{q^2}}.
\end{equation}
\end{Theorem}


\begin{proof}
Consider once again the recurrence relation \eqref{entry12-3term}, and the numerator and denominator polynomials, respectively, $N_k(x)$ and $D_k(x)$. Let 
\begin{align*}
\widehat{N}_k(x) &:= \frac{N_k(x)}{\pqrfac{bq/a}{k}{q^2}} ,
\\
\intertext{and}
\widehat{D}_k(x) &:= \frac{D_k(x)}{\pqrfac{bq/a}{k}{q^2}} 
.
\end{align*}
We see that these polynomials satisfy the recurrence equation
\begin{multline*}
(1-bq^{2k+1}/a) y_{k+1}(x) = ((1-ab)x + (1-ab)q^{2k}) y_k(x) +
\cr ab (1-aq^{2k-1}/b) y_{k-1}(x), 
\end{multline*}
 for  $k > 0$;
with initial values
$$ \widehat{N}_0(x)=0, \widehat{N}_1(x)=\frac{1-ab}{1-bq/a};
\widehat{D}_0(x)=1, \widehat{D}_1(x)=\frac{(1-ab)(x+1)}{1-bq/a}.$$ 

Let $\widehat{N}(t)$ and $\widehat{D}(t)$ denote the respective generating functions. We find that 
\begin{align*}
\widehat{N}(t) & = \frac{(1-ab)t}{(1-\delta_1 t)(1-\delta_2 t)}
\sum_{m=0}^\infty \frac{\pqrfac{-aqt/b, a^2qt}{m}{q^2}}
{\pqrfac{\delta_1 q^2 t, \delta_2 q^2 t}{m}{q^2}} \Big( \frac{bq}{a} \Big)^m, \\
\intertext{and}
\widehat{D}(t) & = \frac{(1-b/aq)}{(1-\delta_1 t)(1-\delta_2 t)}
\sum_{m=0}^\infty \frac{\pqrfac{-aqt/b, a^2qt}{m}{q^2}}
{\pqrfac{\delta_1 q^2 t, \delta_2 q^2 t}{m}{q^2}} \Big( \frac{b}{aq} \Big)^m, \\
\end{align*}
where $\delta_1$ and $\delta_2$ are such that
$$ 1-(1-ab)xt -abt^2 = (1-\delta_1 t)(1-\delta_2 t).$$
We are interested in the value of the continued fraction when $x=1$. But when $x=1$, we have
$$\delta_1 = 1, \delta_2=-ab.$$
If we assume $|ab|<1$, then $t=1$ is the singularity closest to the origin. 

To find the asymptotic behavior of $\widehat{N}_k(1)$, we apply Darboux's method \cite[Th. 1.2.4]{MI2009} with the comparison function
$$N^*(t)=\sum_{k=0}^\infty N^*_k t^k = 
\frac{(1-ab)}{(1-t)(1+ab)}
\sum_{m=0}^\infty \frac{\pqrfac{-aq/b, a^2q}{m}{q^2}}
{\pqrfac{q^2 , -ab q^2 }{m}{q^2}} \Big( \frac{bq}{a} \Big)^m.
$$
Note that $\widehat{N}(t) - N^*(t)$ has a removable singularity at $t=1$, and is thus continuous in the disk $|t|\le r$, for some $r>1$. By Darboux's method, we obtain
$$\widehat{N}(1)= N^*_k+o(r^{-k}).$$
Similarly, we can find an asymptotic formula for $\widehat{D}(1)$. 

In this manner, when $|ab|<1$,  we obtain:
\begin{align*}
\widehat{N}_k(1) \sim \frac{(1-ab)}{1+ab}\rPhis{2}{1}{-{aq}/{b}, a^2q}{-abq^2}{\frac{bq}{a}}{q^2},\\
\intertext{and}
\widehat{D}_k(1) \sim \frac{(1-b/aq)}{1+ab}\rPhis{2}{1}{-{aq}/{b}, a^2q}{-abq^2}{\frac{b}{aq}}{q^2}.
\end{align*}
The continued fraction $H(1)$ is the ratio
$$H(1)=\frac{N_k(1)}{D_k(1)} = \frac{\widehat{N}_k(1)}{\widehat{D}_k(1)},
$$
which is the ratio of the above basic hypergeometric series. 
Finally, we transform both the $_2\phi_1$ series using Heine's transformation formula 
\cite[Equation~(III.1)]{GR90} to complete the proof.
\end{proof}



\begin{Corollary}[Ramanujan's Entry II.16.12]\label{cor:entry12} Let $|q|<1$ and $|ab|<1$. Then, we have
$$
\frac{\pqrfac{a^2q^3, b^2q^3}{\infty}{q^4}}{\pqrfac{a^2q, b^2q}{\infty}{q^4}}
= 
\frac{1}{1-ab}\fplus\frac{(a-bq)(b-aq)}{(1-ab)(1+q^2)}\fplus 
\frac{(a-bq^3)(b-aq^3)}{(1-ab)(1+q^4)}\fplus
\fdots. 
$$
\end{Corollary}
\begin{proof} We recognize that the $_2\phi_1$ series in \eqref{H1} are $D(0)$ and $D(1)$, where $D(s)$ is defined in \eqref{def-ds}. So we have
$$K=H(1)/(1-ab) = \frac{D(1)}{2 D(0)}.$$
Now from \eqref{K-C-relation} and using \eqref{entry12-recursion} (with $s=0$), we have:
\begin{align*}
\frac{1}{C}=\frac{1}{K} -(1-ab) &= 2\frac{D(0)}{D(1)} - (1-ab) \\
& = (1-ab) +
\frac{(a-bq)(b-aq)}{(1+q^{2}){\cfrac{D(1)}{D(2)}}} \cr
&= 
\frac{\displaystyle\rPhis{2}{1}{{b}/{aq}, -b/aq}{-q^2}{a^2q^3}{q^2}}
{ \displaystyle \rPhis{2}{1}{{bq}/{a}, -bq/a}{-q^2}{a^2q}{q^2}},
\end{align*}
where we used \eqref{entry12-pf-step2} to obtain the last equality. Finally, using \eqref{entry12-pf-step1}, we complete the proof.
\end{proof}
It is rather remarkable that we require the  same algebraic manipulations in both our proofs of Entry 12.

\begin{Remarks} \ 
\begin{enumerate}
\item   If $|ab|>1$, then we see that $t=-1/ab$ is the singularity closest to the origin. Again we can find the value of the continued fraction by using Darboux's theorem. Alternatively, perhaps it is easier to see that the continued fraction is equivalent to 
$$\frac{-1/ab}{1-(ab)^{-1}}\fplus\frac{(a^{-1}-b^{-1}q)(b^{-1}-a^{-1}q)}{(1-(ab)^{-1})(1+q^2)}\fplus 
\frac{(a^{-1}-b^{-1}q^3)(b^{-1}-a^{-1}q^3)}{(1-(ab)^{-1})(1+q^4)}\fplus
\fdots .
$$
The value of this (for $|q|<1$) is obtained from  Corollary~\ref{cor:entry12} by replacing $a$ by $1/a$, $b$ by $1/b$ and dividing the products obtained by $-ab$.
\item When $|q|>1$, we see that the original continued fraction of Ramanujan is equivalent to 
$$\frac{1}{1-ab}\fplus\frac{(a-bp)(b-ap)}{(1-ab)(1+p^2)}\fplus 
\frac{(a-bp^3)(b-ap^3)}{(1-ab)(1+p^4)}\fplus
\fdots,
$$
where now $p=1/q$ and $|p|<1$. We now use Entry 12, with $q$ replaced by $p$, to find the value of the continued fraction.
\end{enumerate}
\end{Remarks}

\section{Generalization of Ramanujan's Entry 12}\label{sec:stieltjes}

We now evaluate the generalization of Ramanujan's Entry 12, the
 $J$-fraction \eqref{jfrac-e12}. 
We will again use generating functions and Darboux's method. We will find it convenient to 
scale the polynomials suitably before proceeding. 

We consider the set of polynomials given by
$$P_k(x)=\frac{D_k(\eta x)}{\eta^k(1-ab)^k},$$ 
where $\eta$ is to be prescribed shortly, and $D_k(x)$ are the denominator polynomials satisfying 
\eqref{entry12-3term} and the initial values mentioned above. We see that $P_k(x)$ satisfies the recurrence relation
\begin{multline*}
xP_{k}(x) = P_{k+1}(x) - (q^{2k}/\eta)  P_k(x) 
\cr 
- \frac{ab}{\eta^2(1-ab)^2}(1-bq^{2k-1}/a)(1-aq^{2k-1}/b) P_{k-1}(x).
\end{multline*}
We choose 
$$\eta^2=\frac{-4ab}{(1-ab)^2}$$
to obtain a recurrence relation of the form
\begin{multline}\label{entry12-3term-nevai}
xP_{k}(x) = P_{k+1}(x) +cq^{2k}  P_k(x) 
\cr 
+ \frac{1}{4}(1-bq^{2k-1}/a)(1-aq^{2k-1}/b) P_{k-1}(x), \text{ for } k > 0,
\end{multline}
where 
$$c=-\frac{(1-ab)}{2\sqrt{-ab}}.$$ We assume that $P_k(x)$ satisfies the initial conditions
$$P_0(x)=1, P_1(x)=x-c.$$
We require $a$ and $b$ to be of different signs, so that $c$ is real. 

This is now in the form of the
standard three-term recurrence relation:
\begin{equation} \label{three-term2}
x y_{k}(x) = y_{k+1}(x) + \alpha_k y_k(x) +\beta_k y_{k-1}(x), \text{ for } k>0,
\end{equation}
where $\alpha_k$ and $\beta_k$ are real. 
In our case, we have
$$\alpha_k=cq^{2k} \text{ and } \beta_k= \frac{1}{4}(1-bq^{2k-1}/a)(1-aq^{2k-1}/b).$$ 
Since $a$ and $b$ have opposite signs,  $\beta_k>0$ for $k>0$. So in view of
\cite[Th.\ 2.5.2]{MI2009} we know that the $P_k(x)$ are orthogonal with respect to a positive measure. Further ahead in this section, we will comment further on the measure. 

%

Let $P^*_k(x)$ be the associated polynomials that satisfy 
\eqref{entry12-3term-nevai} with $P^*_0(x)=0$ and $P_1^*(x)=1$. 
The $P^*_k(x)$ are polynomials of 
degree $k-1$, and correspond to the numerator polynomials of the associated continued fraction.  
Let $X(x)$ denote the continued fraction associated with \eqref{entry12-3term-nevai}, that is,  let
\begin{equation}\label{cf-entry12-transformed}
X(x)=\lim_{k\to\infty} \frac{P_k^*(x)}{P_k(x)}.
\end{equation}
Our objective in this section is to compute $X(x)$. 


%
We will take 
 $x=\cos\vartheta$, so that  $$e^{\pm i\vartheta} = x\pm \sqrt{x^2-1}.$$
We choose a 
branch 
 of $\sqrt{x^2-1}$ in such a way that
$$\sqrt{x^2-1} \sim x, \text{ as } x\to\infty,$$
so that
$\left| e^{-i\vartheta}\right|<\left|e^{i\vartheta}\right|$
in the upper half plane, and 
$\left|e^{i\vartheta}\right|<\left|e^{-i\vartheta}\right|$
in the lower half plane. 
We use the notation  $\rho_1=e^{-i\vartheta}$ and $\rho_2=e^{i\vartheta}$. 

\begin{Theorem}\label{th:cf-values} Let $X(x)$ be the continued fraction in \eqref{cf-entry12-transformed}. Let $\rho_1$ and $\rho_2$ be as above.  Let $\gamma_1$ and $\gamma_2$ given by
\begin{equation}\label{gamma}
\gamma_1, \gamma_2 = \frac{aq}{2b}\big(c\pm \sqrt{c^2-1}\big),
\end{equation}
Let $F$ and $G$ be defined as follows:
\begin{align*}
F(\rho) &= \rPhis{2}{1}{2\gamma_1\rho, 2\gamma_2\rho}{q^2\rho^2}{\frac{bq}{a}}{q^2}
, \\
\intertext{and}
G(\rho) &= (1-{b}/{aq})\ 
\rPhis{2}{1}{2\gamma_1\rho, 2\gamma_2\rho}{q^2\rho^2}{\frac{b}{aq}}{q^2}.
\end{align*}
Then  $X(x)$ converges for all complex numbers $x\not\in (-1,1)$, except possibly a finite set of 
points,  and is given by
$$X(x)= 2\rho\frac{F(\rho)}{G(\rho)}, $$
where $\rho$ is given by:
$$\rho = 
\begin{cases}
 \rho_1, & \text{if } \Im(x)>0 , \text{ or } x > 1 \text{ ($x$ real)} \cr
 \rho_2,  & \text{if } \Im(x)<0, \text{ or } x < -1 \text{ ($x$ real)}  \cr
1, & \text{if } x = 1, \cr
 -1, & \text{if } x = -1. 
\end{cases}
$$
\end{Theorem}
\begin{proof}
We will comment on the convergence of $X(x)$ shortly. We first compute asymptotic formulas for $P^*_k(x)$ and $P_k(x)$ in the upper half plane. 

Consider
$$Q_k(x) =\frac{P_k(x)}{\pqrfac{bq/a}{k}{q^2}}.$$
Then $Q_k(x)$ satisfies
\begin{multline}\label{Q-3term}
xQ_{k}(x) =(1-bq^{2k+1}/a) Q_{k+1}(x) +cq^{2k}  Q_k(x) 
\cr 
+ \frac{1}{4}(1-aq^{2k-1}/b) Q_{k-1}(x), \text{ for } k > 0,
\end{multline}
where $$Q_0(x)=1, Q_1(x)=\frac{(x-c)}{(1-bq/a)}.$$

We will apply Darboux's method to find an asymptotic formula for $Q_k(x)$, from which we will get the required formula for $P_k(x)$.
Let $Q(t)$ denote the generating function of $Q_k(x)$, that is, let
$$Q(t):= \sum_{k=0}^\infty Q_k(x)t^k.$$
Multiply \eqref{Q-3term} by $t^{k+1}$ and sum over $k\geq 0$ to find that
$$Q(t) = \frac{(1-b/aq)}{1-xt+t^2/4} + \frac{(1- acqt/b+a^2q^2t^2/4b^2)}{1-xt+t^2/4} 
\Big(\frac{b}{aq}\Big) Q(tq^2).$$
We change the variable by taking $$x=\frac{e^{i\vartheta}+e^{-i\vartheta}}{2} \text{ } (= \cos \vartheta)$$
so
\begin{align*}
1-xt+t^2/4 &= (1-e^{-i\vartheta}t/2)(1-e^{i\vartheta}t/2)\cr
&= (1-\rho_1 t/2)(1-\rho_2 t/2) \text{ (say)}.
\end{align*}
Let $\gamma_1$, $\gamma_2$ be such that
$$(1- acqt/b+a^2q^2t^2/4b^2) = (1-\gamma_1 t)(1-\gamma_2 t),$$
so
$$\gamma_1, \gamma_2 = \frac{aq}{2b}\big(c\pm \sqrt{c^2-1}\big).$$

Using $\rho_1$, $\rho_2$, $\gamma_1$  and $\gamma_2$, we can write the $q$-difference equation for $Q(t)$ in the form 
\begin{align*}
Q(t) &= \frac{(1-b/aq)}{(1-\rho_1 t/2)(1-\rho_2 t/2)} 
+
 \frac{(1-\gamma_1 t)(1-\gamma_2 t)}
{(1-\rho_1 t/2)(1-\rho_2 t/2) } \Big(\frac{b}{aq}\Big) 
Q(tq^2) .
\end{align*}
On iteration, we obtain
\begin{equation*} 
Q(t) = \frac{(1-b/aq)}{(1-\rho_1 t/2)(1-\rho_2 t/2)} 
\sum_{k=0}^{\infty} \frac{\pqrfac{\gamma_1 t, \gamma_2 t}{k}{q^2} }
{\pqrfac{\rho_1 q^2 t/2, \rho_2 q^2t/2 }{k}{q^2}}
\Big( \frac{b}{aq}\Big)^{k}.
\end{equation*}
Note that the poles of $Q(t)$ are at  
$$t=2\rho_1, 2\rho_1/q^2, 2\rho_1/q^4,\dots; \text{ and }
t=2\rho_2, 2\rho_2/q^2, 2\rho_2/q^4\dots.$$
Since $0<|q|<1 $, the pole nearest to $t=0$ is at $t=2\rho_1$ or
$t=2\rho_2$. 

Next let
$$Q_k^*(x)=\frac{P^*_k(x)}{\pqrfac{bq/a}{k}{q^2}},$$ 
so $Q^*_k(x)$ satisfies \eqref{Q-3term} with the initial conditions: $$Q_0^*(x) =0 \text{ and }
 Q_1^*(x) =1/(1-bq/a).$$ 

As before, 
let $\gamma_1$ and $\gamma_2$ be given by \eqref{gamma}. 
The generating function 
of
 $Q_k^{*}$ is given by
\begin{equation*}
Q^*(t) = \frac{t}{(1-\rho_1 t/2)(1-\rho_2 t/2)} 
\sum_{k=0}^{\infty} \frac{\pqrfac{\gamma_1 t, \gamma_2 t}{k}{q^2} }
{\pqrfac{\rho_1 q^2t/2, \rho_2 q^2t/2 }{k}{q^2}}
\Big( \frac{bq}{a}\Big)^{k}.
\end{equation*}

In the upper half-plane, the singularity nearest the origin is at $t=2\rho_1$. 
Let $S^*(t)$ be the series
\begin{equation*}%
{S^*}(t)= 
\frac{2\rho_1}{(1-\rho_1^2)(1-\rho_2 t/2)}
\sum_{m=0}^{\infty} 
\frac{\pqrfac{2\gamma_1\rho_1, 2\gamma_2\rho_1}{m}{q^2}}
{\pqrfac{q^2, q^2\rho_1^2 }{m}{q^2}}
\Big( \frac{bq}{a}\Big)^{m}.
\end{equation*}
and note that $Q^*(t)-{S^*}(t)$ has a removable singularity at $t=2\rho_1.$
By Darboux's method, 
we have
$$Q_k^*(x) \sim 
\frac{2\rho_1 \rho_2^k}{2^k(1-\rho_1^2) }
\sum_{m=0}^{\infty} 
\frac{\pqrfac{2\gamma_1\rho_1, 2\gamma_2\rho_1}{m}{q^2}}
{\pqrfac{q^2, q^2\rho_1^2 }{m}{q^2}}
\Big( \frac{bq}{a}\Big)^{m}
= \frac{2\rho_1 \rho_2^k}{2^k(1-\rho_1^2) } F(\rho_1) \text{ (say)}. $$
Similarly, considering the generating function of $Q_k(x)$  in the upper half plane, we find that
$$Q_k (x) \sim 
\frac{\rho_2^k(1-b/aq)}{2^k(1-\rho_1^2) }
\sum_{m=0}^{\infty} 
\frac{\pqrfac{2\gamma_1\rho_1, 2\gamma_2\rho_1}{m}{q^2}}
{\pqrfac{q^2, q^2\rho_1^2 }{m}{q^2}}
\Big( \frac{b}{aq}\Big)^{m}
= \frac{\rho_2^k}{2^k(1-\rho_1^2) } G(\rho_1) \text{ (say)}. $$

Thus in the upper half-plane, we find that
$$X(x)=\lim_{k\to\infty} \frac{P_k^*(x)}{P_k(x)} =
\lim_{k\to\infty} \frac{Q_k^*(x)\pqrfac{bq/a}{k}{q^2}}{Q_k(x) \pqrfac{bq/a}{k}{q^2}} 
=
2\rho_1\frac{F(\rho_1)}{G(\rho_1)}.$$
The same calculation works for $x>1$, when $x$ is real. In the statement of our theorem, we write 
$F$ and $G$ in terms of the $_r\phi_s$ notation. 

In the lower half-plane, since $t=2\rho_2$ is the singularity nearest to the origin, 
a similar calculation yields
$$X(x)=\lim_{k\to\infty} \frac{P_k^*(x)}{P_k(x)} =2\rho_2\frac{F(\rho_2)}{G(\rho_2)}.$$ 
This is also valid for $x<-1$, for $x$ real.

For $x=1$, we find that the dominating term of the comparison function for the generating function of $P^*_k(1)$ is given by
\begin{equation*}%
{S^*}(t)= 
\frac{2}{(1-t/2)^2}
\sum_{m=0}^{\infty} 
\frac{\pqrfac{2\gamma_1\rho_1, 2\gamma_2\rho_1}{m}{q^2}}
{\pqrfac{q^2, q^2\rho_1^2 }{m}{q^2}}
\Big( \frac{bq}{a}\Big)^{m}.
\end{equation*}
(There is another term of the form $A/(1-t/2)$, but that does not have any contribution.) By Darboux's theorem, we find that
$$Q^*_k (1) \sim \frac{2(k+1)}{2^k}F(1).$$
Similarly, we have
$$Q_k (1) \sim \frac{(k+1)}{2^k}G(1).$$ 
Thus 
$$X(1) = 2 \frac{F(1)}{G(1)}.$$
A similar calculation is required for $x=-1$. 
\end{proof}

Recall that the $P_k(x)$ satisfy the three term recurrence relation \eqref{three-term2}, with $\alpha_k$ and $\beta_k$ real, bounded, and $\beta_k>0$ for $k>0$. This implies that the polynomials $P_k(x)$ are orthogonal with respect to a positive measure $\mu$. Further, the support of $\mu$ is bounded. From Blumenthal's theorem \cite[Theorem IV-3.5, p.~117]{chihara1978} it follows that $\mu$ has an absolutely continuous component $\mu^\prime$. 
We close this section with some further remarks on the orthogonality measure.

\begin{Remarks}\ 
\begin{enumerate}
\item We can apply Nevai's theorem \cite[Theorem 40, p.~143]{PN1979} to obtain an expression for $\mu^\prime$ supported on $(-1,1)$.
This requires the use of Darboux's theorem as above. For an example of a very similar calculation, in the context of another continued fraction of Ramanujan, see \cite[\S 3]{BI2019a}. 
\item Nevai's theorem further states that if $\mu$ has a discrete part, it lies outside $(-1,1)$. Since the measure is bounded, there are only a finite set of isolated mass points of $\mu$ outside $(-1,1)$. 
\item Markov's theorem \cite[Theorem 2.6.2]{MI2009} implies that the value of the continued fraction $X(x)$ is the Stieltjes transform of the measure for $x\not \in (a,b)$, where $(a,b)$ is the true interval
of orthogonality. 
 This implies that the continued fraction $X(x)$ converges for all complex numbers $x\not\in (-1,1)$, except possibly a finite set of points, as asserted by Theorem~\ref{th:cf-values}.
\item As done in \cite[\S 4]{BI2019a}, we can obtain another expression for $\mu$ by inverting the Stieltjes transform, which is equal to $X(x)$ in Theorem~\ref{th:cf-values}. 
\item See also the remarks after Theorem~4.1 in \cite{BI2019a}.
\end{enumerate}
The details of the analysis mentioned above are very similar to the analogous calculations in \cite{BI2019a}, and are omitted.
\end{Remarks}
\subsection*{Acknowledgments} This work was done at the sidelines of conferences and summer schools organized by the members of the Orthogonal Polynomials and Special Functions (OPSF) group of SIAM. We thank the organizers of the following: Summer research institute on $q$-series, (July-Aug 2018), Chern Institute of Mathematics, Nankai University, Tianjin, PR China and OPSFA2019, Hagenberg, Austria. The research of the first author was supported by a grant of the Austrian Science Fund (FWF): F 50-N15.

%



\end{document}